\documentclass[12pt]{amsart}
\usepackage{enumerate}
\usepackage[latin1]{inputenc}
\usepackage[catalan]{babel}
\usepackage{graphicx}
\usepackage{psfrag}
\usepackage{amsmath}
\usepackage{amsfonts}
\usepackage{amssymb}

\newcounter{lletres}

\newtheorem{Talpha}[lletres]{Teorema}
\newtheorem{Lalpha}[lletres]{Lemma}
\newtheorem{Calpha}[lletres]{Corol-lari}

\newcommand{\RR}{{\mathbb R}}
\newcommand{\ZZ}{{\mathbb Z}}
\newcommand{\NN}{{\mathbb N}}

\newcommand{\Zplus}{{\mathbb {Z^+}}}

\numberwithin{equation}{section}

\begin{document}
\title{Ortoedres amb longitud d'arestes enteres / Cuboids with integer length edges}
\author{Daniel Blasi Babot}
\address{Daniel Blasi Babot}
\email{dblasibabot@gmail.com}

\begin{abstract}
L'objectiu d'aquest article és estudiar el nombre d'ortoedres diferents $\mathcal{O}(N)$ que es poden formar amb una quantitat arbitrària $N$ de cubs. A l'article s'obté un métode iteratiu per a calcular el valor de $\mathcal{O}(N)$ per $n\in\NN$ qualsevol. Utilitzant el métode s'obté una fórmula explícita quan $N$ és producte de dues potències de nombres primers diferents. També s'estudia el cas bidimensional i es dóna una fórmula general que determina el nombre de rectangles diferents que es poden formar amb una quantitat arbitrària $N$ de quadrats.
\end{abstract}

\maketitle

\section{Introducció}

Els ortoedres han estat llargament estudiats. Fixem-nos per exemple en l'estudi dels ortoedres racionals \cite{L1, L2}. Els ortoedres racionals estan caracteritzats per 7 nombres racionals positius (3 arestes diferents, 3 diagonals de les cares diferents i la diagonal interior). Els bricks d'Euler són ortoedres amb les arestes i les diagonals de les cares enteres \cite{D}. Si a més a més la diagonal principal de l'ortoedre és entera aleshores parlem d'un ortoedre perfecte.

En el nostre cas no treballem amb ortoedres perfectes. Ens interessa estudiar un cas més general d'ortoedres en què l'úncia restricció és que les longituds de les arestes siguin enteres. Donat un nombre $N\in\NN$ voldríem saber quants ortoedres diferents de volum $N$ existeixen que tinguin la longitud de les seves arestes enteres.\footnote{El problema d'estudiar el nombre d'ortoedres diferents que es poden formar amb $p^n\cdot q^m$ cubs el va proposar el Jordi Font Gonzàlez a partir d'una proposta d'investigació amb policubs.}

Donat un conjunt $\mathcal{A}$ qualsevol, definim 
  $$\sharp\mathcal{A} =\text{nombre d'elements del conjunt }\mathcal{A}.$$

Estudiarem primer el cas bidimensional.

\section{Recompte de rectangles}

Donat un nombre $N\in\NN$ qualsevol voldríem saber quants rectangles diferents es poden obtenir de manera que els seus costats tinguin longitud entera i el seu volum sigui $N$.\\

Sigui $N\in\NN$ i $x\in\RR$, definim

$$\mathcal{R}(N)=\sharp\{\text{Rectangles diferents d'àrea }N \text{ i longitud dels costats enteres}\},$$

$$d(N) =\sharp\{\text{Divisors del nombre } N\},$$

$$\Zplus=\{n\in\ZZ\text{ amb }n>0\},$$

$$[x] = \sup\{z\in\ZZ\text{ tal que }z\leq x\}.$$\\

Donat un nombre $N\in\NN$, el següent resultat ens dóna una fórmula per $\mathcal{R}(N)$.\\

\begin{Talpha}
Sigui $N=p_1^{m_1}\cdots p_n^{m_n}$ amb $\{p_i\}_{i=1,...,n}$ nombres primers diferents i $m_i\in\Zplus$ per $i=1,...,n.$ Aleshores
$$\mathcal{R}(N)=\left\{
\begin{array}{lll}
\displaystyle\frac{(m_1+1)\cdots (m_n+1)+1}{2}&\text{ si }&N=t^2\text{ per algun }t\in\ZZ\\\\
\displaystyle\frac{(m_1+1)\cdots (m_n+1)}{2}&\text{ si }&N=t^2+1\text{ per algun }t\in\ZZ
\end{array}
\right.
$$\\
\end{Talpha}

\begin{proof}
Fixem-nos que les dimensions d'un rectangle qualsevol d'àrea $N$ venen determinades per dos divisors $a,b$ de $N$ amb $a\cdot b=N.$ Així doncs el teorema quedarà provat si veiem que  $d(N)=(m_1+1)\cdots (m_n+1)$. \\

Fixem-nos que $d(p_1^{m_1})=m_1+1,\;$ ja que els divisors de $p_1^{m_1}$ són $1,\;p_1,\;p_1^2,...,p_1^{m_1}$. \\

Observem que $d(p_1^{m_1}\cdot p_2^{m_2})=d(p_1^{m_1})\cdot d(p_2^{m_2}).$ En efecte, qualsevol divisor de $p_1^{m_1}\cdot p_2^{m_2}$ és producte d'un divisor de $p_1^{m_1}$ per un divisor de $p_2^{m_2}$ ja que $p_1$ i $p_2$ són nombres primers diferents. Així doncs $d(p_1^{m_1}\cdot p_2^{m_2})=(m_1+1)\cdot(m_2+1).$\\

Per recurrència podem extendre el raonament i obtenim el resultat desitjat.\\
\end{proof}

\section{Recompte d'ortoedres}

Donat un nombre $N\in\NN$ voldríem saber quants ortoedres diferents podem construir amb longitud de les arestes enteres i volum $N$.\\

Definim 
$$\mathcal{O}(N)=\sharp\{\text{Ortoedres diferents de volum }N\text{ i longitud de les arestes enteres}\}.$$

Direm que $(A,B,C)$ és un ortoedre de volum $N$ si $A\cdot B\cdot C=N.$ Pensarem que les tripletes $(A,B,C),\; (A,C,B),\; (B,A,C),\; (B,C,A),\; (C,A,B)$ i $(C,B,A)$ representen el mateix ortoedre. \\



Així doncs, el problema que volem estudiar es pot reformular de la següent manera: donat un nombre $N\in\NN$ qualsevol, voldríem saber quantes tripletes diferents $(A,B,C)$ existeixen amb $A,\; B,\; C\in \Zplus,$ $A\leq B\leq C$ i $A\cdot B\cdot C=N.$ \\

El següent resultat ens dóna un métode recursiu per a calcular $\mathcal{O}(N),$ partint de la descomposició en factors primers d'un nombre $N\in\NN$ qualsevol.

\begin{Talpha}\label{maintheorem}
Si $N=p_1^{n_1}\cdot ...\cdot p_k^{n_k},$ amb $n_j=6s_j+\alpha_j,$ per $j=1,...,k$, aleshores
$$\mathcal{O}(N)=f_{\alpha^k}(s^k)+g_{\alpha^k}(s^k)+h_{\alpha^k}(s^k),$$
on $\alpha^k=(\alpha_1,...,\alpha_k),$ $s^k=(s_1,...,s_k)$ i $f_{\alpha^k}(s^k),\;g_{\alpha^k}(s^k),\;h_{\alpha^k}(s^k)$ són les funcions definides a \eqref{eqnfk,gk,hk}.
\end{Talpha}

A la secció \ref{section exemples} trobarem un exemple de com aplicar aquest resultat per calcular $\mathcal{O}(N)$ en un cas concret. Així mateix, el següent resultat mostra una fórmula explícita per  $\mathcal{O}(N)$, quan $N$ és producte de dues potències de nombres primers diferents, utilitzant el mètode recursiu del Teorema \ref{maintheorem}.\\

\begin{Talpha}\label{theorempq}
Sigui $N=p^n\cdot q^m$ amb $p\neq q$ nombres primers i $n,m\in\Zplus,$ sigui
    $$w(n,m)=\frac{2n^2+2m^2+12nm+3n^2m+3nm^2+n^2m^2}{24},$$
aleshores
\begin{itemize}
 \item si $n\equiv 0\mod{6}$,
  $$\def\arraystretch{2.2}\begin{array}{rll}
     \mathcal{O}(N)=&\displaystyle\frac{24+12n+12m}{24}+w(n,m)
        &\text{ si } m\equiv 0\mod(6)\\
     \mathcal{O}(N)=&\displaystyle\frac{(n+2)\cdot(m+1)\cdot(nm+2n+m+5)}{24}&\text{ si }
          m\equiv 1\text{ o } 5\mod(6)\\  
     \mathcal{O}(N)=&\displaystyle\frac{(n+2)\cdot(m+2)\cdot(nm+n+m+4)}{24}&\text{ si }
          m\equiv 2\text{ o } 4\mod(6)\\ 
     \mathcal{O}(N)=&\displaystyle \frac{18+9n+12m}{24}+w(n,m)&\text{ si }
          m\equiv 3 \mod(6)\\  
  \end{array}$$
 \item si $n\equiv 1\text{ o }n\equiv 5\mod{6},$
  $$\def\arraystretch{2.2}\begin{array}{rll}
     \mathcal{O}(N)=&\displaystyle\frac{(n+1)\cdot(m+2)\cdot(nm+n+2m+5)}{24}&
         \text{ si }m\equiv 0,\; 2\text{ o } 4\mod(6)\\
     \mathcal{O}(N)=&\displaystyle\frac{(n+1)\cdot(m+1)\cdot(nm+2n+2m+7)}{24}&
         \text{ si }m\equiv 1,\; 3\text{ o } 5\mod(6)\\
  \end{array}$$

 \item si $n\equiv 2\text{ o }n\equiv 4\mod{6},$
  $$\def\arraystretch{2.2}\begin{array}{rll}
     \mathcal{O}(N)=&\displaystyle\frac{(n+2)\cdot(m+2)\cdot(nm+n+m+4)}{24}&
         \text{ si }m\equiv 0,\; 2\text{ o } 4\mod(6)\\
     \mathcal{O}(N)=&\displaystyle\frac{(n+2)\cdot(m+1)\cdot(nm+2n+m+5)}{24}&
         \text{ si }m\equiv 1,\; 3\text{ o } 5\mod(6)\\
  \end{array}$$

 \item si $n\equiv 3\mod{6},$
  $$\def\arraystretch{2.2}\begin{array}{rll}
     \mathcal{O}(N)=&\displaystyle \frac{18+12n+9m}{24}+w(n,m)&\text{ si }m\equiv 0 \mod(6)\\
     \mathcal{O}(N)=&\displaystyle\frac{(n+1)\cdot(m+1)\cdot(nm+2n+2m+7)}{24}&
         \text{ si }m\equiv 1\text{ o } 5\mod(6)\\
     \mathcal{O}(N)=&\displaystyle\frac{(n+1)\cdot(m+2)\cdot(nm+n+2m+5)}{24}&
         \text{ si }m\equiv 2\text{ o } 4\mod(6)\\
     \mathcal{O}(N)=&\displaystyle \frac{15+9n+9m}{24}+w(n,m)&\text{ si }
          m\equiv 3 \mod(6)\\
  \end{array}$$
   
\end{itemize}
\end{Talpha}

El cas $m=0$ del teorema anterior ens dóna una fórmula per $\mathcal{O}(N)$ quan $N$ és una potència d'un nombre primer.

\begin{Calpha}\label{corollaryp^n}
Sigui $N=p^n$ amb $p$ un nombre primer i $n\in\Zplus,$ aleshores
$$\displaystyle\mathcal{O}(N)=\left\{
\begin{array}{lll}
\displaystyle\frac{n^2+6n+12}{12}&\text{ si }&n\equiv 0\mod{ 6 }\\\\
\displaystyle\frac{n^2+6n+5}{12}&\text{ si }&n\equiv 1\mod{ 6 } \text{ o }\;n\equiv 5\mod{ 6 }\\\\
\displaystyle\frac{n^2+6n+8}{12}&\text{ si }&n\equiv 2\mod{ 6 } \text{ o }\;n\equiv 4\mod{ 6 }\\\\
\displaystyle\frac{n^2+6n+9}{12}&\text{ si }&n\equiv 3\mod{ 6 }
\end{array}
\right.
$$\\
\end{Calpha}

Si $N$ és producte de nombres primers diferents, el següent resultat ens dóna una fórmula per  $\mathcal{O}(N).$

\begin{Talpha}
Sigui $N=p_1 \cdots p_n$ amb $\{p_i\}_{i=1,...,n}$ nombres primers diferents, aleshores
$$\mathcal{O}(N)=\frac{3^{n-1}+1}{2}$$\\
\end{Talpha}

\begin{proof}
Volem comptar el nombre de tripletes $(A,B,C)$ amb $A,B,C\in\Zplus$, $A\cdot B \cdot C=N$ i $A\leq B\leq C.$ \\

Com que $N=p_1\cdot...\cdot p_n,$ per comptar totes les tripletes possibles pensem que cadascun dels factors $p_i$ per $i=1,\cdots,n$ pot estar inclós en la factorització de $A,$ de $B$ o de $C$. Tenim doncs $3^n$ possibilitats, però algunes d'elles no compleixen $A\leq B\leq C$. \\

Cada ortoedre $(A,B,C)$  l'estem comptant 6 vegades si els tres valors $A,\; B$ i $C$ són diferents: $(A,B,C)$, $(A,C,B)$, $(B,A,C)$, $(B,C,A)$, $(C,A,B)$ i $(C,B,A)$. \\

Si a dues lletres els correspon el valor 1 tenim tres repeticions: $(1,1,N)$, $(1,N,1)$ i $(N,1,1)$. Per tant,
$$\mathcal{O}(N)=\frac{3^n-3}{6}+1=\frac{3^{n-1}-1}{2}+1=\frac{3^{n-1}+1}{2}$$
i el resultat queda provat.
\end{proof}

\section{Prova del teorema \ref{theorempq}}\label{sectionprooftheorempq}

Si $N=p^n\cdot q^m$ amb $p\neq q$ nombres pirmers i $n,m\in\NN$ llavors 
\begin{equation}\label{eqnO(N)}
\mathcal{O}(N)=\sharp\{(p^i\cdot q^k,p^j\cdot q^l,p^{n-i-j}\cdot q^{m-k-l})\}
\end{equation}
amb $i,j,k,l\in\Zplus$, $i+j\leq n,$ $k+l\leq m$.\\

Donats dos ortoedres $(A,B,C)$ i $(D,E,F)$ definim 
	$$\begin{array}{ll}
	(A,B,C)\otimes(D,E,F)=\{&\!\!\!\!(A\cdot D,B\cdot E,C\cdot F),\;(A\cdot D,B\cdot F,C\cdot E),\\
	&\!\!\!\!(A\cdot E,B\cdot D,C\cdot F),\;	(A\cdot E,B\cdot F,C\cdot D),\\
	&\!\!\!\!(A\cdot F,B\cdot D,C\cdot E),\;(A\cdot F,B\cdot E,C\cdot D)\;\,\},
	\end{array}$$
és a dir, el conjunt d'ortoedres que tenen en cada coordenada el producte d'una coordenada del primer ortoedre i una coordenada del segon ortoedre sense repetir-les.\\

\begin{Lalpha}\label{lemmaproducte}
Donats $A,B,C,D,E,F\in\NN,$ considerem el conjunt d'ortoedres
$$\mathcal{A}=\{(A,B,C)\otimes(D,E,F)\}.$$
Aleshores
$$\sharp \mathcal{A} = \left\{\begin{array}{lll}
		6 & \text{si}&A,B,C\text{ diferents i }\;D,E,F\text{ diferents}\\
		3 & \text{si}&\text{dos elements de }A,B,C\text{ iguals i }\;D,E,F\text{ diferents o viceversa}\\
		2 & \text{si}&\text{dos elements de }A,B,C\text{ iguals i }\;\text{dos elements de }D,E,F\text{ iguals}\\
		1 & \text{si}&A=B=C\text{ o bé }\;D=E=F.\\	
					  \end{array}\right.$$\\
\end{Lalpha}

\begin{proof}
La prova és immediata tenint en compte que 
$$(A,B,C),\;(A,C,B),\;(B,A,C),\;(B,C,A),\;(C,A,B)\text{ i }(C,B,A)$$
representen el mateix ortoedre, per qualssevol $A,B,C\in\NN.$\\
\end{proof}

Donats dos conjunts $\mathcal{A}$ i $\mathcal{B}$ d'ortoedres definim el conjunt
$$\mathcal{A}\otimes\mathcal{B}=\Big\{(A,B,C)\otimes(D,E,F)\Big\}_{
          \begin{array}{l}
              (A,B,C)\in\mathcal{A}\\
              (D,E,F)\in\mathcal{B}\;.
          \end{array}}$$

Siguin

$$\mathcal{A}=\Big\{(p^i,p^j,p^{n-i-j})\Big\}_{\begin{array}{l}
					i=0,...,[\frac{n}{3}]\\
					j=i,...,\left[\frac{n-i}{2}\right]
			   \end{array}}$$

i

$$\mathcal{B}=\Big\{(q^k,q^l,q^{m-k-l})\Big\}_{\begin{array}{l}
					k=0,...,[\frac{m}{3}]\\
					l=k,...,\left[\frac{m-k}{2}\right]
			   \end{array}},$$\\

dos conjunts d'ortoedres, aleshores \eqref{eqnO(N)} es pot reescriure com

    $$\mathcal{O}(N)=\sharp \{\mathcal{A}\otimes\mathcal{B}\},$$\\

o equivalentment, \\

  $$\mathcal{O}(N)=\sum_{(A,B,C)\in\mathcal{A}}\;\;\sum_{(D,E,F)\in\mathcal{B}}\sharp\{(A,B,C)\otimes(D,E,F)\}.$$\\

Si $p$ és un nombre primer i $n\in\NN$, podem escriure $n=6s+\alpha$ amb $s\in\NN$ i $\alpha=0,1,2,3,4,5.$ Definim\\

$$\mathcal{A}^{\alpha}(s)=\Big\{(p^i,p^j,p^{6s+\alpha-i-j})\Big\}_{\def\arraystretch{1.5}
			   \begin{array}{l}
					i=0,...,\left[\frac{6s+\alpha}{3}\right]\\
					j=i,...,\left[\frac{6s+\alpha-i}{2}\right].
			   \end{array}}$$

Dividim el conjunt $\mathcal{A}^{\alpha}(s)$ en dos subconjunts disjunts
$$\mathcal{A}^{\alpha}(s)=\mathcal{A}_o^{\alpha}(s)\uplus\mathcal{A}_e^{\alpha}(s),$$

on

$$\mathcal{A}_o^{\alpha}(s)=\Big\{(p^{2u},p^j,p^{6s+\alpha-2u-j})\Big\}_{\def\arraystretch{1.3}
               \begin{array}{l}
					u=0,...,s\\
					j=2u,...,3s-u+\left[\frac{\alpha}{2}\right]
			   \end{array}}$$
			   
correspon al conjunt de ternes de $\mathcal{A}^{\alpha}(s)$ amb la primera coordenada una potència d'exponent parell i 

$$\mathcal{A}_e^{\alpha}(s)=\Big\{(p^{2u+1},p^j,p^{6s+\alpha-2u-1-j})\Big\}_{\def\arraystretch{1.3}
               \begin{array}{l}
					u=0,...,\left[\frac{6s+\alpha-3}{6}\right]\\
					j=2u+1,...,3s-u+\left[\frac{\alpha-1}{2}\right]
			   \end{array}}$$
			   
correspon al conjunt de ternes de $\mathcal{A}^{\alpha}(s)$ amb la primera coordenada una potència d'exponent senar. \\

Considerem una altra partició disjunta de $\mathcal{A}^{\alpha}(s)$:

\begin{equation}\label{eqnA_i}
    \mathcal{A}^{\alpha}(s)=
    \mathcal{A}_1^{\alpha}(s)\uplus\mathcal{A}_2^{\alpha}(s)\uplus\mathcal{A}_3^{\alpha}(s),
\end{equation}

on

  $$\def\arraystretch{2}
  \begin{array}{l}
     \mathcal{A}_1^{\alpha}(s)=\big\{(p^i,p^i,p^i)\big\}_{3i=6s+\alpha}\\
     \mathcal{A}_2^{\alpha}(s)=\mathcal{A}_{2'}^{\alpha}(s)\uplus\mathcal{A}_{2''}^{\alpha}(s)
         \text{ amb}\\
          \quad\quad\mathcal{A}_{2'}^{\alpha}(s)=\big\{(p^i,p^i,p^{6s+\alpha-2i})\big\}_
                 {0\leq i<6s+\alpha-2i}\\
          \quad\quad\mathcal{A}_{2''}^{\alpha}(s)=\big\{(p^i,p^j,p^j)\big\}_
                 {i+2j=6s+\alpha,\;0\leq i<j}\\
     \mathcal{A}_3^{\alpha}(s)=\big\{(p^i,p^j,p^{6s+\alpha-i-j})\big\}_{0\leq i<j<6s+\alpha-i-j}    .
  \end{array}$$\\

Fixem-nos que $\mathcal{A}_1^{\alpha}(s)$ correspon a les ternes de $\mathcal{A}^{\alpha}(s)$ amb les tres coordenades iguals, $\mathcal{A}_2^{\alpha}(s)$ a les ternes de $\mathcal{A}^{\alpha}(s)$ amb dues coordenades iguals i una diferent i $\mathcal{A}_3^{\alpha}(s)$ correspon a les ternes de $\mathcal{A}^{\alpha}(s)$ amb les tres coordenades diferents. Definim 

\begin{equation}\label{eqn fgh}
f_{\alpha}(s)=\sharp\mathcal{A}_1^{\alpha}(s),\quad g_{\alpha}(s)=\sharp\mathcal{A}_2^{\alpha}(s)\quad\text{i}\quad h_{\alpha}(s)=\sharp\mathcal{A}_3^{\alpha}(s).
\end{equation}\\

\begin{Lalpha}\label{Lemafgh}
Donat un nombre primer $p$ i un nombre $n\in\NN,$ escrivim $n=6s+\alpha$ amb $\alpha\in\NN$ i $s=0,1,2,3,4,5$. Sigui
$$\mathcal{A}^{\alpha}(s)=\Big\{(p^i,p^j,p^{6s+\alpha-i-j})\Big\}_{\def\arraystretch{1.5}
			   \begin{array}{l}
					i=0,...,\left[\frac{6s+\alpha}{3}\right]\\
					j=i,...,\left[\frac{6s+\alpha-i}{2}\right].
			   \end{array}}.$$
Aleshores,

\begin{equation}\label{eqnfgh,k=1}
\def\arraystretch{1.5}\begin{array}{rlll}
      \text{si}\quad\alpha=0,&f_0(s)=1&g_0(s)=3s&h_0(s)=3s^2\\
      \text{si}\quad\alpha=1,&f_1(s)=0&g_1(s)=3s+1&h_1(s)=3s^2+s\\
      \text{si}\quad\alpha=2,&f_2(s)=0&g_2(s)=3s+2&h_2(s)=3s^2+2s\\
      \text{si}\quad\alpha=3,&f_3(s)=1&g_3(s)=3s+1&h_3(s)=3s^2+3s+1\\
      \text{si}\quad\alpha=4,&f_4(s)=0&g_4(s)=3s+3&h_4(s)=3s^2+4s+1\\
      \text{si}\quad\alpha=5,&f_5(s)=0&g_5(s)=3s+3&h_5(s)=3s^2+5s+2,\\
  \end{array}
\end{equation}\\
$\text{on}\quad 
f_{\alpha}(s)=\sharp\mathcal{A}_1^{\alpha}(s),\quad 
g_{\alpha}(s)=\sharp\mathcal{A}_2^{\alpha}(s)\quad \text{i}\quad 
h_{\alpha}(s)=\sharp\mathcal{A}_3^{\alpha}(s),\quad\text{per}\quad\alpha=0,...,5.$\\

%
%
%
\end{Lalpha}

\begin{proof}

Podem suposar sense pèrdua de generalitat qu $s>0$ ja que en aquest cas el resultat és obvi. Anem a diferenciar la prova pels diferents valors de $\alpha$.

\begin{itemize}

  \item Si $\alpha=0,$ és a dir, $n=6\alpha,$

      $$\mathcal{A}^{0}(s)=\Big\{(p^i,p^j,p^{6s-i-j})\Big\}_{\def\arraystretch{1.3}
	      \begin{array}{l}
					i=0,...,2s\\
					j=i,...,\left[\frac{6s-i}{2}\right].
	      \end{array}}$$

Ara,

$$\mathcal{A}_o^{0}(s)=\Big\{(p^{2u},p^j,p^{6s-2u-j})\Big\}_{
               \begin{array}{l}
					u=0,...,s\\
					j=2u,...,3s-u
			   \end{array}}$$
			   
i per tant

$$ \displaystyle\sharp\mathcal{A}_o^{0}(s)=\sum_{u=0}^s\;\sum_{j=2u}^{3s-u}1=
               \sum_{u=0}^s 3s-3u+1=3s^2+4s+1-\frac{3s^2+3s}{2}.$$

Fixem-nos que 

$$\mathcal{A}_e^{0}(s)=\Big\{(p^{2u+1},p^j,p^{6s-2u-1-j})\Big\}_{
               \begin{array}{l}
					u=0,...,s-1\\
					j=2u+1,...,3s-u-1
			   \end{array}}$$

i per tant

  $$\displaystyle\sharp\mathcal{A}_e^{0}(s)=\sum_{u=0}^{s-1}\;\sum_{j=2u+1}^{3s-u-1}1=
               \sum_{u=0}^{s-1} 3s-3u+1=3s^2-s+\frac{-3s^2+3s}{2}.$$

Així doncs tenim que

$$\sharp{A}^0(s)=\sharp{A}_o^0(s)+\mathcal{A}_e^0(s)=3s^2+3s+1.$$

Considerant la partició \eqref{eqnA_i} de $\mathcal{A}^0(s)$ 
$$\def\arraystretch{2}\begin{array}{l}
  \mathcal{A}_1^{0}(s)=\big\{(p^i,p^i,p^i)\big\}_{3i=6s}=
           \big\{(p^{2s},p^{2s},p^{2s})\big\}\\
           
  \mathcal{A}_{2'}^{0}(s)=\big\{(p^i,p^i,p^{6s-2i})\big\}_
                 {i=0,...,2s-1}\\
                 
  \mathcal{A}_{2''}^{0}(s)=\big\{(p^i,p^j,p^j)\big\}_
                 {i+2j=6s,\;0\leq i<j}=\big\{(p^{2u},p^j,p^j)\big\}_{u=0,...,s-1}\\
                 
  \mathcal{A}_3^{0}(s)=\big\{(p^i,p^j,p^{6s-i-j})\big\}_{0\leq i<j<6s-i-j}\\
\end{array}$$
tenim que
$$\def\arraystretch{2}\begin{array}{l}
       \sharp\mathcal{A}^0_1(s)=1,\\
       \sharp\mathcal{A}^0_2(s)=\sharp\mathcal{A}^0_{2'}(s)+\sharp\mathcal{A}^0_{2''}(s)=(2s)+(s)=3s\\
       \sharp\mathcal{A}^0_3(s)=
         \sharp\mathcal{A}^0(s)-\sharp\mathcal{A}^0_1(s)-\sharp\mathcal{A}^0_2(s)=
         (3s^2+3s+1)-(1)-(3s)=3s^2,
\end{array}$$

per tant 

$$f_0(s)=1,\quad g_0(s)=3s\quad\text{i}\quad h_0(s)=3s^2.$$\\

 \item Si $\alpha=1,$ és a dir, $n=6\alpha+1,$

      $$\mathcal{A}^{1}(s)=\Big\{(p^i,p^j,p^{6s+1-i-j})\Big\}_{\def\arraystretch{1.3}
	      \begin{array}{l}
					i=0,...,2s\\
					j=i,...,\left[\frac{6s+1-i}{2}\right].
	      \end{array}}$$

Ara,

$$\mathcal{A}_o^{1}(s)=\Big\{(p^{2u},p^j,p^{6s+1-2u-j})\Big\}_{
               \begin{array}{l}
					u=0,...,s\\
					j=2u,...,3s-u
			   \end{array}}$$
			   
i per tant

$$ \displaystyle\sharp\mathcal{A}_o^{1}(s)=\sum_{u=0}^s\;\sum_{j=2u}^{3s-u}1=
               \sum_{u=0}^s 3s-3u+1=3s^2+4s+1-\frac{3s^2+3s}{2}.$$

Fixem-nos que 

$$\mathcal{A}_e^{1}(s)=\Big\{(p^{2u+1},p^j,p^{6s-2u-j})\Big\}_{
               \begin{array}{l}
					u=0,...,s-1\\
					j=2u+1,...,3s-u
			   \end{array}}$$

i per tant

  $$\displaystyle\sharp\mathcal{A}_e^{1}(s)=\sum_{u=0}^{s-1}\;\sum_{j=2u+1}^{3s-u}1=
               \sum_{u=0}^{s-1} 3s-3u=3s^2+\frac{-3s^2+3s}{2}.$$

Així doncs tenim que

$$\sharp{A}^1(s)=\sharp{A}_o^1(s)+\mathcal{A}_e^1(s)=3s^2+4s+1.$$

Considerant la partició \eqref{eqnA_i} de $\mathcal{A}^1(s)$ 
$$\def\arraystretch{2}\begin{array}{l}
  \mathcal{A}_1^{1}(s)=\big\{(p^i,p^i,p^i)\big\}_{3i=6s+1}=\emptyset\\
           
  \mathcal{A}_{2'}^{1}(s)=\big\{(p^i,p^i,p^{6s+1-2i})\big\}_
                 {i=0,...,2s}\\
                 
  \mathcal{A}_{2''}^{1}(s)=\big\{(p^i,p^j,p^j)\big\}_
                 {i+2j=6s+1,\;0\leq i<j}=\big\{(p^{2u+1},p^j,p^j)\big\}_{u=0,...,s-1}\\
                 
  \mathcal{A}_3^{1}(s)=\big\{(p^i,p^j,p^{6s+1-i-j})\big\}_{0\leq i<j<6s+1-i-j}\\
\end{array}$$
tenim que
$$\def\arraystretch{2}\begin{array}{l}
       \sharp\mathcal{A}^1_1(s)=0,\\
       \sharp\mathcal{A}^1_2(s)=\sharp\mathcal{A}^1_{2'}(s)+\sharp\mathcal{A}^1_{2''}(s)=(2s+1)+(s)=3s+1\\
       \sharp\mathcal{A}^1_3(s)=
         \sharp\mathcal{A}^1(s)-\sharp\mathcal{A}^1_1(s)-\sharp\mathcal{A}^1_2(s)=
         3s^2+s,
\end{array}$$

per tant 

$$f_1(s)=0,\quad g_1(s)=3s+1\quad\text{i}\quad h_1(s)=3s^2+s.$$\\

 \item Si $\alpha=2,$ és a dir, $n=6\alpha+2,$

      $$\mathcal{A}^{2}(s)=\Big\{(p^i,p^j,p^{6s+2-i-j})\Big\}_{\def\arraystretch{1.3}
	      \begin{array}{l}
					i=0,...,2s\\
					j=i,...,\left[\frac{6s+2-i}{2}\right].
	      \end{array}}$$

Ara,

$$\mathcal{A}_o^{2}(s)=\Big\{(p^{2u},p^j,p^{6s+2-2u-j})\Big\}_{
               \begin{array}{l}
					u=0,...,s\\
					j=2u,...,3s+1-u
			   \end{array}}$$
			   
i per tant

$$ \displaystyle\sharp\mathcal{A}_o^{2}(s)=\sum_{u=0}^s\;\sum_{j=2u}^{3s+1-u}1=
               \sum_{u=0}^s 3s-3u+2=3s^2+5s+2-\frac{3s^2+3s}{2}.$$

Fixem-nos que 

$$\mathcal{A}_e^{2}(s)=\Big\{(p^{2u+1},p^j,p^{6s+1-2u-j})\Big\}_{
               \begin{array}{l}
					u=0,...,s-1\\
					j=2u+1,...,3s-u
			   \end{array}}$$

i per tant

  $$\displaystyle\sharp\mathcal{A}_e^{2}(s)=\sum_{u=0}^{s-1}\;\sum_{j=2u+1}^{3s-u}1=
               \sum_{u=0}^{s-1} 3s-3u=3s^2+\frac{-3s^2+3s}{2}.$$

Així doncs tenim que

$$\sharp{A}^2(s)=\sharp{A}_o^2(s)+\mathcal{A}_e^2(s)=3s^2+5s+2.$$

Considerant la partició \eqref{eqnA_i} de $\mathcal{A}^2(s)$ 
$$\def\arraystretch{2}\begin{array}{l}
  \mathcal{A}_1^{2}(s)=\big\{(p^i,p^i,p^i)\big\}_{3i=6s+2}=\emptyset\\
           
  \mathcal{A}_{2'}^{2}(s)=\big\{(p^i,p^i,p^{6s+2-2i})\big\}_
                 {i=0,...,2s}\\
                 
  \mathcal{A}_{2''}^{2}(s)=\big\{(p^i,p^j,p^j)\big\}_
                 {i+2j=6s+2,\;0\leq i<j}=\big\{(p^{2u},p^j,p^j)\big\}_{u=0,...,s}\\
                 
  \mathcal{A}_3^{2}(s)=\big\{(p^i,p^j,p^{6s+2-i-j})\big\}_{0\leq i<j<6s+2-i-j}\\
\end{array}$$
tenim que
$$\def\arraystretch{2}\begin{array}{l}
       \sharp\mathcal{A}^2_1(s)=0,\\
       \sharp\mathcal{A}^2_2(s)=\sharp\mathcal{A}^2_{2'}(s)+\sharp\mathcal{A}^2_{2''}(s)=(2s+1)+(s+1)=3s+2\\
       \sharp\mathcal{A}^2_3(s)=
         \sharp\mathcal{A}^2(s)-\sharp\mathcal{A}^2_1(s)-\sharp\mathcal{A}^2_2(s)=
         3s^2+2s,
\end{array}$$

per tant 

$$f_2(s)=0,\quad g_2(s)=3s+2\quad\text{i}\quad h_2(s)=3s^2+2s.$$\\

 \item Si $\alpha=3,$ és a dir, $n=6\alpha+3,$

      $$\mathcal{A}^{3}(s)=\Big\{(p^i,p^j,p^{6s+3-i-j})\Big\}_{\def\arraystretch{1.3}
	      \begin{array}{l}
					i=0,...,2s+1\\
					j=i,...,\left[\frac{6s+3-i}{2}\right].
	      \end{array}}$$

Ara,

$$\mathcal{A}_o^{3}(s)=\Big\{(p^{2u},p^j,p^{6s+3-2u-j})\Big\}_{
               \begin{array}{l}
					u=0,...,s\\
					j=2u,...,3s+1-u
			   \end{array}}$$
			   
i per tant

$$ \displaystyle\sharp\mathcal{A}_o^{3}(s)=\sum_{u=0}^s\;\sum_{j=2u}^{3s+1-u}1=
               \sum_{u=0}^s 3s-3u+2=3s^2+5s+2-\frac{3s^2+3s}{2}.$$

Fixem-nos que 

$$\mathcal{A}_e^{3}(s)=\Big\{(p^{2u+1},p^j,p^{6s+2-2u-j})\Big\}_{
               \begin{array}{l}
					u=0,...,s\\
					j=2u+1,...,3s-u+1
			   \end{array}}$$

i per tant

  $$\displaystyle\sharp\mathcal{A}_e^{3}(s)=\sum_{u=0}^{s}\;\sum_{j=2u+1}^{3s-u+1}1=
               \sum_{u=0}^{s} 3s-3u+1=3s^2+4s+1-\frac{3s^2+3s}{2}.$$

Així doncs tenim que

$$\sharp{A}^3(s)=\sharp{A}_o^3(s)+\mathcal{A}_e^3(s)=3s^2+6s+3.$$

Considerant la partició \eqref{eqnA_i} de $\mathcal{A}^3(s)$ 
$$\def\arraystretch{2}\begin{array}{l}
  \mathcal{A}_1^{3}(s)=\big\{(p^i,p^i,p^i)\big\}_{3i=6s+3}=
      \big\{(p^{2s+1},p^{2s+1},p^{2s+1})\big\}\\
           
  \mathcal{A}_{2'}^{3}(s)=\big\{(p^i,p^i,p^{6s+3-2i})\big\}_
                 {i=0,...,2s}\\
                 
  \mathcal{A}_{2''}^{3}(s)=\big\{(p^i,p^j,p^j)\big\}_
                 {i+2j=6s+3,\;0\leq i<j}=\big\{(p^{2u+1},p^j,p^j)\big\}_{u=0,...,s-1}\\
                 
  \mathcal{A}_3^{3}(s)=\big\{(p^i,p^j,p^{6s+3-i-j})\big\}_{0\leq i<j<6s+3-i-j}\\
\end{array}$$
tenim que
$$\def\arraystretch{2}\begin{array}{l}
       \sharp\mathcal{A}^3_1(s)=1,\\
       \sharp\mathcal{A}^3_2(s)=\sharp\mathcal{A}^3_{2'}(s)+\sharp\mathcal{A}^3_{2''}(s)=(2s+1)+(s)=3s+1\\
       \sharp\mathcal{A}^3_3(s)=
         \sharp\mathcal{A}^3(s)-\sharp\mathcal{A}^3_1(s)-\sharp\mathcal{A}^3_2(s)=
         3s^2+3s+1,
\end{array}$$

per tant 

$$f_3(s)=1,\quad g_3(s)=3s+1\quad\text{i}\quad h_3(s)=3s^2+3s+1.$$\\

 \item Si $\alpha=4,$ és a dir, $n=6\alpha+4,$

      $$\mathcal{A}^{4}(s)=\Big\{(p^i,p^j,p^{6s+4-i-j})\Big\}_{\def\arraystretch{1.3}
	      \begin{array}{l}
					i=0,...,2s+1\\
					j=i,...,\left[\frac{6s+4-i}{2}\right].
	      \end{array}}$$

Ara,

$$\mathcal{A}_o^{4}(s)=\Big\{(p^{2u},p^j,p^{6s+4-2u-j})\Big\}_{
               \begin{array}{l}
					u=0,...,s\\
					j=2u,...,3s+2-u
			   \end{array}}$$
			   
i per tant

$$ \displaystyle\sharp\mathcal{A}_o^{4}(s)=\sum_{u=0}^s\;\sum_{j=2u}^{3s+2-u}1=
               \sum_{u=0}^s 3s-3u+3=3s^2+6s+3-\frac{3s^2+3s}{2}.$$

Fixem-nos que 

$$\mathcal{A}_e^{4}(s)=\Big\{(p^{2u+1},p^j,p^{6s+3-2u-j})\Big\}_{
               \begin{array}{l}
					u=0,...,s\\
					j=2u+1,...,3s-u+1
			   \end{array}}$$

i per tant

  $$\displaystyle\sharp\mathcal{A}_e^{4}(s)=\sum_{u=0}^{s}\;\sum_{j=2u+1}^{3s-u+1}1=
               \sum_{u=0}^{s} 3s-3u+1=3s^2+4s+1-\frac{3s^2+3s}{2}.$$

Així doncs tenim que

$$\sharp{A}^4(s)=\sharp{A}_o^4(s)+\mathcal{A}_e^4(s)=3s^2+7s+4.$$

Considerant la partició \eqref{eqnA_i} de $\mathcal{A}^4(s)$ 
$$\def\arraystretch{2}\begin{array}{l}
  \mathcal{A}_1^{4}(s)=\big\{(p^i,p^i,p^i)\big\}_{3i=6s+4}=\emptyset\\
           
  \mathcal{A}_{2'}^{4}(s)=\big\{(p^i,p^i,p^{6s+4-2i})\big\}_
                 {i=0,...,2s+1}\\
                 
  \mathcal{A}_{2''}^{4}(s)=\big\{(p^i,p^j,p^j)\big\}_
                 {i+2j=6s+4,\;0\leq i<j}=\big\{(p^{2u},p^j,p^j)\big\}_{u=0,...,s}\\
                 
  \mathcal{A}_3^{4}(s)=\big\{(p^i,p^j,p^{6s+4-i-j})\big\}_{0\leq i<j<6s+4-i-j}\\
\end{array}$$
tenim que
$$\def\arraystretch{2}\begin{array}{l}
       \sharp\mathcal{A}^4_1(s)=0,\\
       \sharp\mathcal{A}^4_2(s)=\sharp\mathcal{A}^4_{2'}(s)+\sharp\mathcal{A}^4_{2''}(s)=(2s+2)+(s+1)=3s+3\\
       \sharp\mathcal{A}^4_3(s)=
         \sharp\mathcal{A}^4(s)-\sharp\mathcal{A}^4_1(s)-\sharp\mathcal{A}^4_2(s)=
         3s^2+4s+1,
\end{array}$$

per tant 

$$f_4(s)=0,\quad g_4(s)=3s+3\quad\text{i}\quad h_4(s)=3s^2+4s+1.$$\\

 \item Si $\alpha=5,$ és a dir, $n=6\alpha+5,$

      $$\mathcal{A}^{5}(s)=\Big\{(p^i,p^j,p^{6s+5-i-j})\Big\}_{\def\arraystretch{1.3}
	      \begin{array}{l}
					i=0,...,2s+1\\
					j=i,...,\left[\frac{6s+5-i}{2}\right].
	      \end{array}}$$

Ara,

$$\mathcal{A}_o^{5}(s)=\Big\{(p^{2u},p^j,p^{6s+5-2u-j})\Big\}_{
               \begin{array}{l}
					u=0,...,s\\
					j=2u,...,3s+2-u
			   \end{array}}$$
			   
i per tant

$$ \displaystyle\sharp\mathcal{A}_o^{5}(s)=\sum_{u=0}^s\;\sum_{j=2u}^{3s+2-u}1=
               \sum_{u=0}^s 3s-3u+3=3s^2+6s+3-\frac{3s^2+3s}{2}.$$

Fixem-nos que 

$$\mathcal{A}_e^{5}(s)=\Big\{(p^{2u+1},p^j,p^{6s+4-2u-j})\Big\}_{
               \begin{array}{l}
					u=0,...,s\\
					j=2u+1,...,3s-u+2
			   \end{array}}$$

i per tant

  $$\displaystyle\sharp\mathcal{A}_e^{5}(s)=\sum_{u=0}^{s}\;\sum_{j=2u+1}^{3s-u+2}1=
               \sum_{u=0}^{s} 3s-3u+2=3s^2+5s+2-\frac{3s^2+3s}{2}.$$

Així doncs tenim que

$$\sharp{A}^5(s)=\sharp{A}_o^5(s)+\mathcal{A}_e^5(s)=3s^2+8s+5.$$

Considerant la partició \eqref{eqnA_i} de $\mathcal{A}^5(s)$ 
$$\def\arraystretch{2}\begin{array}{l}
  \mathcal{A}_1^{5}(s)=\big\{(p^i,p^i,p^i)\big\}_{3i=6s+5}=\emptyset\\
           
  \mathcal{A}_{2'}^{5}(s)=\big\{(p^i,p^i,p^{6s+5-2i})\big\}_
                 {i=0,...,2s+1}\\
                 
  \mathcal{A}_{2''}^{5}(s)=\big\{(p^i,p^j,p^j)\big\}_
                 {i+2j=6s+5,\;0\leq i<j}=\big\{(p^{2u+1},p^j,p^j)\big\}_{u=0,...,s}\\
                 
  \mathcal{A}_3^{5}(s)=\big\{(p^i,p^j,p^{6s+5-i-j})\big\}_{0\leq i<j<6s+5-i-j}\\
\end{array}$$
tenim que
$$\def\arraystretch{2}\begin{array}{l}
       \sharp\mathcal{A}^5_1(s)=0,\\
       \sharp\mathcal{A}^5_2(s)=\sharp\mathcal{A}^5_{2'}(s)+\sharp\mathcal{A}^5_{2''}(s)=(2s+2)+(s+1)=3s+3\\
       \sharp\mathcal{A}^5_3(s)=
         \sharp\mathcal{A}^5(s)-\sharp\mathcal{A}^5_1(s)-\sharp\mathcal{A}^5_2(s)=
         3s^2+5s+2,
\end{array}$$

per tant 

$$f_5(s)=0,\quad g_4(s)=3s+3\quad\text{i}\quad h_5(s)=3s^2+5s+2.$$\\

\end{itemize}
\end{proof}

A continuació detallem la prova del teorema \ref{theorempq}.\\

\begin{proof}

Si $N=p^{n}\cdot q^m$ amb $p\neq q$ nombres primers i $n,m\in\NN,$ escrivim $n=6s+\alpha$ i $m=6t+\beta$ amb $s,t\in\NN,\;$ $\alpha,\beta=0,1,2,3,4,5.$ Llavors,

$$\begin{array}{rl}
	\mathcal{O}(N)=&\sharp\;\big\{\;\{\mathcal{A}_1^{\alpha}(s)\uplus
	     \mathcal{A}_2^{\alpha}(s)\uplus\mathcal{A}_3^{\alpha}(s)\}\;\otimes\;
	     \{\mathcal{A}_1^{\beta}(t)\uplus\mathcal{A}_2^{\beta}(t)\uplus
	     \mathcal{A}_3^{\beta}(t)\}\;\Big\}\\
    =&\displaystyle\sum_{i=1}^3\sum_{j=1}^3\sharp\{\mathcal{A}_i^{\alpha}(s)\;\otimes\;\mathcal{A}_j^{\beta}(t)\},
  \end{array}$$

ja que els conjunts són disjunts.\\

Utilitzant el Lema \ref{lemmaproducte} tenim que 

\begin{equation}\label{eqnO(N)ambfgh}
  \begin{array}{rl}
    \mathcal{O}(N)=&6\cdot h_{\alpha}(s)\cdot h_{\beta}(t)+
         3\cdot(g_{\alpha}(s)\cdot h_{\beta}(t)+h_{\alpha}(s)\cdot g_{\beta}(t))+\\
     &+\;2\cdot g_{\alpha}(s)\cdot g_{\beta}(t)+f_{\alpha}(s)\cdot f_{\beta(t)}+\\
     &+\;f_{\alpha}(s)\cdot(g_{\beta}(t)+h_{\beta}(t))+
       f_{\beta}(t)\cdot(g_{\alpha}(s)+h_{\alpha}(s)).
  \end{array}
\end{equation}\\

Utilitzant les fórmules obtingudes al Lema \ref{Lemafgh} el resultat és immediat.

\end{proof}

\section{Prova del Teorema \ref{maintheorem}}\label{proofmaintheorem}

Donat $N\in\NN$ considerem la seva descomposició factorial $N=p_1^{n_1}\cdot...\cdot p_n^{n_k},$ amb $p_1,...,p_n$ nombres primers diferents i $n_1,...,n_k\in\Zplus.$ 
Escrivim 
$$n_i=6s_i+\alpha_i\quad \text{per certs}\quad s_i\in\NN\quad \text{i}\quad \alpha_i=0,...,5\quad \text{per}\quad i=1,...,k$$ 
i definim els vectors
$$\alpha^i=(\alpha_1,...,\alpha_i)\quad\text{i}\quad s^i=(s_1,...,s_i),\quad\text{per}\quad i=1,...,k.$$
Definim les funcions
$$f_{\alpha^1}(t^1):=f_{\alpha_1}(t),\quad g_{\alpha^1}(t^1):=g_{\alpha_1}(t),\quad g_{\alpha^1}(t^1):=g_{\alpha_1}(t),$$
on $f_{\alpha_1}(t),\;g_{\alpha_1}(t),\;h_{\alpha_1}(t)$ són les funcions definides a \eqref{eqn fgh} i $t^1=(t)$ és un vector unidimensional qualsevol de coordenada $t$.\\

Observem que 
\begin{itemize}
\item $f_{\alpha_1}(s_1)$ correspon al nombre de tripletes\\
       $\def\arraystretch{1.3}\begin{array}{l}
         \quad\quad (A,A,A)\text{ amb }\text{ amb }A\in\Zplus,\;A^3=p_1^{6s_1+\alpha_1}.
        \end{array}$
\item $g_{\alpha_1}(s_1)$ correspon al nombre de tripletes\\
       $\def\arraystretch{1.3}\begin{array}{l}
         \quad\quad(A,A,B)\text{ amb }\text{ amb }A,B\in\Zplus,\;A<B\text{ i }A^2\cdot B=p_1^{6s_1+\alpha_1}\text{ o bé}\\
         \quad\quad(A,B,B)\text{ amb }\text{ amb }A,B\in\Zplus,\;A<B\text{ i }A\cdot B^2=p_1^{6s_1+\alpha_1}.
       \end{array}$
\item $h_{\alpha_1}(s_1)$ correspon al nombre de tripletes\\ 
       $\def\arraystretch{1.3}\begin{array}{l}
         \quad\quad (A,B,C)\text{ amb }\text{ amb }A<B<C\in\Zplus,\text{ i }A\cdot B\cdot C=p_1^{6s_1+\alpha_1}.
        \end{array}$
\end{itemize}

Definim de manera recursiva les funcions
\begin{equation}\label{eqnfk,gk,hk}
  \begin{array}{ll}
    f_{\alpha^i}(t^i)=&f_{\alpha^{i-1}}(t^{i-1})\cdot f_{\alpha_i}(t_i),\\\\
    g_{\alpha^i}(t^i)=&f_{\alpha^{i-1}}(t^{i-1})\cdot g_{\alpha_i}(t_i)+
                       g_{\alpha^{i-1}}(t^{i-1})\cdot (f_{\alpha_i}(t_i)+g_{\alpha_i}(t_i)).\\\\
    h_{\alpha^i}(t^i)=&6\cdot h_{\alpha^{i-1}}(t^{i-1})\cdot h_{\alpha_i}(t_i)+
                       g_{\alpha^{i-1}}(t^{i-1})\cdot g_{\alpha_i}(t_i)+\\\\
                      &f_{\alpha^{i-1}}(t^{i-1})\cdot h_{\alpha_i}(t_i)+
                       h_{\alpha^{i-1}}(t^{i-1})\cdot f_{\alpha_i}(t_i)+\\\\
                      &3\cdot(g_{\alpha^{i-1}}(t^{i-1})\cdot h_{\alpha_i}(t_i)+
                       h_{\alpha^{i-1}}(t^{i-1})\cdot g_{\alpha_i}(t_i)). 
  \end{array}
\end{equation}

per $i=2,...,k$, on $t^i=(t_1,...,t_i)$ és un vector de dimensió $i$. Observem que 
\begin{itemize}
\item $f_{\alpha^i}(s^i)$ corresponen al nombre de tripletes\\
       $\def\arraystretch{1.3}\begin{array}{l}
         \quad\quad (A,A,A)\text{ amb }A\in\Zplus,\; A^3=p_1^{6s_1+\alpha_1}\cdot ... \cdot p_i^{6s_i+\alpha_i}.
        \end{array}$
\item $g_{\alpha^i}(s^i)$ correspon al nombre de tripletes\\
       $\def\arraystretch{1.3}\begin{array}{l}
         \quad\quad(A,A,B)\text{ amb }A<B\in\Zplus\text{ i }A^2\cdot B=p_1^{6s_1+\alpha_1}\cdot ... \cdot p_i^{6s_i+\alpha_i}\text{ o bé}\\
         \quad\quad(A,B,B)\text{ amb }A<B\in\Zplus\text{ i }A\cdot B^2=p_1^{6s_1+\alpha_1}\cdot ... \cdot p_i^{6s_i+\alpha_i}.
       \end{array}$
\item $h_{\alpha^i}(s^i)$ correspon al nombre de tripletes\\ 
       $\def\arraystretch{1.3}\begin{array}{l}
         \quad\quad (A,B,C)\text{ amb }A<B<C\in\Zplus\text{ i }A\cdot B\cdot C=p_1^{6s_1+\alpha_1}\cdot ... \cdot p_i^{6s_i+\alpha_i}.
        \end{array}$
\end{itemize}       

Per tant $\mathcal{O}(N)=f_{\alpha^k}(s^k)+g_{\alpha^k}(s^k)+h_{\alpha^k}(s^k).$

\section{Exemple}\label{section exemples}

Anem a veure com calcular $\mathcal{O}(N)$ en un cas concret.

\begin{itemize}

\item Si $N=2^6\cdot 3^7\cdot 5^6 = 2.187.000.000$ aleshores 
     $$n_1=6,\quad n_2=6+1\quad\text{i}\quad n_3=6.$$ 
Per tant, 
  $$s_1=1,\;\alpha_1=0,\;s_2=1,\;\alpha_2=1,\; s_3=1,\text{ i }\alpha_3=0.$$
Així doncs 
   $$\def\arraystretch{1.3}\begin{array}{l}
      f_{\alpha^1}(s^1)=f_{\alpha_1}(s_1)=f_0(1)=1,\\
      g_{\alpha^1}(s^1)=g_{\alpha_1}(s_1)=g_0(1)=3\cdot 1=3,\\
      h_{\alpha^1}(s^1)=h_{\alpha_1}(s_1)=h_0(1)=3\cdot 1^2=3.\\
     \end{array}$$
També 
   $$\def\arraystretch{1.3}\begin{array}{l}
     f_{\alpha_2}(s_2)=f_1(1)=0,\\
     g_{\alpha_2}(s_2)=g_1(1)=3\cdot 1 + 1 = 4,\\
     h_{\alpha_2}(s_2)=h_1(1)=3\cdot 1^2 + 1 = 4,\\
     f_{\alpha_3}(s_3)=f_0(1)=1,\\
     g_{\alpha_3}(s_3)=g_0(1)=3\cdot 1 = 3,\\
     h_{\alpha_3}(s_3)=h_0(1)=3\cdot 1^2 = 3.\\
   \end{array}$$
   
Ara, 
  $$\alpha^2=(\alpha_1,\alpha_2)=(0,1)\quad\text{i}\quad s^2=(s_1,s_2)=(1,1)$$
i llavors 
   $$\def\arraystretch{1.3}\begin{array}{rl}
      f_{\alpha^2}(s^2)=&\!\!\!\!f_{\alpha^1}(s^1)\cdot f_{\alpha_2}(s_2)=1\cdot 0=0,\\
      g_{\alpha^2}(s^2)=&\!\!\!\!f_{\alpha^{1}}(s^{1})\cdot g_{\alpha_2}(s_2)+
                       g_{\alpha^{1}}(s^{1})\cdot (f_{\alpha_2}(s_2)+g_{\alpha_2}(s_2))=\\
                      =&\!\!\!\! 1\cdot 4 + 3\cdot(0+4)=16\\
      h_{\alpha^2}(s^2)=&6\cdot h_{\alpha^{1}}(s^{1})\cdot h_{\alpha_2}(s_2)+
                         g_{\alpha^{1}}(s^{1})\cdot g_{\alpha_2}(s_2)+\\
                        &f_{\alpha^{1}}(s^{1})\cdot h_{\alpha_2}(s_2)+
                         h_{\alpha^{1}}(s^{1})\cdot f_{\alpha_2}(s_2)+\\
                        &3\cdot(g_{\alpha^{1}}(s^{1})\cdot h_{\alpha_2}(s_2)+
                         h_{\alpha^{1}}(s^{1})\cdot g_{\alpha_2}(s_2))=\\
                       =&6\cdot 3\cdot 4 + 3\cdot 4+1\cdot 4+1\cdot 0+3\cdot(3\cdot 4+3\cdot 4)=\\
                       =&72+12+4+0+72=160
   \end{array}$$
Finalment considerem
  $$\alpha^3=(\alpha_1,\alpha_2,\alpha_3)=(0,1,0)\quad\text{i}\quad s^3=(s_1,s_2,s_3)=(1,1,1)$$
i llavors
   $$\def\arraystretch{1.3}\begin{array}{rl}
      f_{\alpha^3}(s^3)=&\!\!\!\!f_{\alpha^2}(s^2)\cdot f_{\alpha_3}(s_3)=0\cdot 1=0,\\
      g_{\alpha^3}(s^3)=&\!\!\!\!f_{\alpha^2}(s^2)\cdot g_{\alpha_3}(s_3)+
                       g_{\alpha^{2}}(s^{2})\cdot (f_{\alpha_3}(s_3)+g_{\alpha_3}(s_3))=\\
                      =&\!\!\!\! 0\cdot 3 + 16\cdot(1+3)=64\\
      h_{\alpha^3}(s^3)=&6\cdot h_{\alpha^2}(s^2)\cdot h_{\alpha_3}(s_3)+
                         g_{\alpha^2}(s^2)\cdot g_{\alpha_3}(s_3)+\\
                        &f_{\alpha^2}(s^2)\cdot h_{\alpha_3}(s_3)+
                         h_{\alpha^2}(s^2)\cdot f_{\alpha_3}(s_3)+\\
                        &3\cdot(g_{\alpha^2}(s^2)\cdot h_{\alpha_3}(s_3)+
                         h_{\alpha^2}(s^2)\cdot g_{\alpha_3}(s_3))=\\
                       =&6\cdot 160\cdot 3 + 16\cdot 3+0\cdot 3+160\cdot 1+3\cdot(16\cdot 3+160\cdot 3)=\\
                       =&2880+48+0+160+1584=4672.
   \end{array}$$
Per tant,

  $$\mathcal{O}(N)=f_{\alpha^3}(s^3)+g_{\alpha^3}(s^3)+h_{\alpha^3}(s^3)=0+64+4672=4736.$$  
\end{itemize}

\bibliographystyle{plain}
\bibliography{Policubs}

\end{document}